\documentclass[reqno,12pt]{amsart}
\usepackage[english]{babel} 
\usepackage[latin1]{inputenc} 
\usepackage[T1]{fontenc} 
\usepackage{indentfirst} 
\usepackage{mathpazo}
\usepackage{amsmath,amssymb,amsfonts,geometry}
\usepackage[colorlinks=true,urlcolor=blue,
citecolor=red,linkcolor=blue,linktocpage,pdfpagelabels,bookmarksnumbered,bookmarksopen]{hyperref}

\usepackage{enumitem}
\newcommand{\C}{\mathbb C}\newcommand{\Z}{\mathbb Z}
\newcommand{\U}{\mathsf{U}}

\newcommand{\Imaginaary}{\Im m}  
  \newcommand{\e}{\operatorname{e}}
\newcommand{\Vol}{\mathbf{V}\!\operatorname{ol}} 

\newcommand{\overbar}[1]{\mkern 1.5mu\overline{\mkern-1.5mu#1\mkern-1.5mu}\mkern 1.5mu} \newcommand{\scalforms}[1]{\left<#1\right>_{\Omega^p}}
\numberwithin{equation}{section}  \makeatletter\@addtoreset{equation}{section} \makeatletter \@addtoreset{equation}{section}
\makeatletter\@addtoreset{equation}{subsection}
 
\newcommand{\abs}[1]{\left\vert#1\right\vert} \newcommand{\set}[1]{\left\{#1\right\}}\newcommand{\scal}[1]{\left<#1\right>}
 
\usepackage{fancybox}
\usepackage{tikz-cd}
\usepackage{tikz}
\usepackage{calc}

\newcommand{\mixedSAutomorphic}{{\mathcal{M}}^{\nu,\mu}_{\rho,\tau}}
\newcommand{\mixedSDelta}{{\Delta}_{\tau}^{\nu,\mu}}

\newcommand{\frakh}{\mathfrak{h}}

\newcommand{\JiMuNu}{J^{\nu,\mu}_{\rho,\tau}} 

\newcommand{\B}{B_{\tau}^{\mu,\nu}}

\newcommand{\varphim}{\varphi_{\tau}^{\nu,\mu}}

\newtheorem {theorem}{Theorem}[section] \newtheorem {definition}[theorem]{Definition} 

 \newtheorem {keylemma}[theorem]{Key Lemma}
\newtheorem {proposition}[theorem]{Proposition} \newtheorem {remark}[theorem]{Remark} \newtheorem {example}[theorem]{Example}
\newtheorem {corollary}[theorem]{Corollary}
\begin{document}
	\title[A lifting theorem for planar mixed automorphic functions and applications]  
	{ A lifting theorem for planar mixed automorphic functions and applications
	}
	\author{Aymane El Fardi} 
	\address{(A.E.) Univ Paris Est Creteil, CNRS, LAMA, F-94010 Creteil, France}
		\address{$\&$ Univ Gustave Eiffel, LAMA, F-77447 Marne-la-Vallée, France}
 
		\email{a.elfardi@gmail.com }
		
		\author{Allal Ghanmi} \author{Lahcen Imlal}	
	\address{(A.G. $\&$ L.I) Analysis, P.D.E. $\&$ Spectral Geometry, Lab. M.I.A.-S.I.,  CeReMAR, \newline
		Department of Mathematics, P.O. Box 1014,  Faculty of Sciences, 
		 Mohammed V University in Rabat, Morocco}
	\email{allal.ghanmi@um5.ac.ma}
	\email{imlalimlal@gmail.com}
	\maketitle
	\begin{abstract}
We deal with the concrete spectral analysis of an invariant magnetic Schr\"odinger operator acting on one dimensional $L^2$-mixed automorphic functions with respect to given equivariant pair $ (\rho,\tau) $  and given discrete subgroup of the semi-direct group $ \U(1)\ltimes\C$.  This will be carried out by means of a lifting theorem to the classical automorphic functions associated with specific pseudo-character. We also provided a partial characterization of the equivariant pairs relative to our setting and discuss possible generalization to higher dimensions.
	\end{abstract}

\section{Introduction}
Mixed automorphic functions (MAFs) generalize elliptic modular forms and can be tracked back to Stiller \cite{stiller-special}. They arise naturally in the description of holomorphic forms of highest degree on family of abelian varieties presented as Riemann surfaces 
\cite{hunt1985mixed}, and are defined on a manifold $M$, on which a discrete group $ \Gamma $ acts properly and discontinuously, as functions satisfying the functional equation  $ f(\gamma\cdot z) = J(\gamma,z)J(\rho(\gamma),\tau(z)) f(z)$, for given equivariant pair $(\rho,\tau)$ (see Definition~\ref{def:equivariant} below)
 and automorphic factor $J(\gamma,z)$.
The case of the hyperbolic plane was extensively studied by Min Ho Lee and his co-workers, whom analyzed different aspects of associated theory. See \cite{lee-mixed} and the bibliographical references therein. 
 The planar euclidean case $M= \C$ with respect to full rank lattice $L$ of $ (\C,+) $ and given automorphic factor 
$ j^\nu(\gamma,z):= \chi(\gamma) e^{i\nu\Imaginaary\scal{z,\gamma}},$  
 was considered in  \cite{ghanmi2011characterization,elgourari2011spectral}.
 Here $\nu >0$, $ \scal{z,w}=z\overbar{w} $ is the usual hermitian inner product on $\C $,  and $  \chi$  is a given mapping (pseudo-character) on $L$ with values in the unitary group $U(1)$ satisfying $ \chi(\gamma \gamma')  = e^{i \nu \Imaginaary \scal{\gamma' ,\gamma }} \chi(\gamma) \chi(\gamma') $ for all  $ \gamma,\gamma'\in\Gamma $.

 In the present paper, we emphasize to continue in this direction by investigating the general case of arbitrary discrete subgroup of $\Gamma$ the semi-direct group $\U(1)\ltimes \C$ of the unitary group $U(1)$ and $\C$.
Our main purpose is the concrete spectral properties of the invariant magnetic Schr\"odinger operator 
$$H_{\theta^{\nu,\mu}_\tau} =(d+ \operatorname{ext} \ \theta^{\nu,\mu}_\tau)^{*}(d+\operatorname{ext} \ \theta^{\nu,\mu}_\tau)$$
associated to the potential vector
\begin{equation}\label{eq:theta2:Chap5}
	\theta^{\nu,\mu}_\tau(z) : ={-} \frac{\nu}{2}(\bar z dz -  zd\bar z) -\frac{\mu}{2}(\overline{\tau}d\tau - \tau d\overline{\tau})  , 
\end{equation}
and acting on $\Gamma$-mixed automorphic functions with respect to  given equivariant pair $(\rho,\tau)$ 
belonging to the Hilbert space $ L^{2,\nu}_{\mu,\tau}(\C/\Gamma)$ of  square integrable functions with respect to the scalar product 
$$ \scal{f,g}_{\nu,\mu,\tau} := \int_{\C/\Gamma} f(z) \overline{g(z)} e^{-\nu|z|^2 - \mu (|E\tau(z)|^2  - |\overline{E}\tau(z)|^2 )} dg_\tau(z)$$
involving the special gaussian weight function 
$$
dg_\tau(z):= e^{-\nu|z|^2 - \mu (|E\tau(z)|^2  - |\overline{E}\tau(z)|^2 )} dxdy; \, z=x+iy.$$
Here $E =z \frac{\partial}{\partial z}$ is the complex Euler operator.
 Mainly, we show that the construction in  \cite{elgourari2011spectral} remains 
valid and the spectrum of $H_{\theta^{\nu,\mu}_\tau}$ is discrete and coincides with the Landau levels of the Landau Hamiltonian studied in \cite{ghanmiintissar}. We also derive the explicit closed expression of their $L^2$-eigenprojector when acting on the free Hilbert space  $L^2(\C,d\lambda)$.

It should be mentioned here that a direct description of such properties are hard to handle  by solving the associated eigenvalue problem. However, we establish a lifting theorem, generalizing the particular one obtained in \cite{elgourari2011spectral}, and providing an isomorphic transformation between the eigenvalue problem of our constructed invariant Schr\"odinger operator on MAFs and the eigenvalue problem for the Landau Hamiltanian with constant magnetic field on the space of classical automorphic functions. 
As a side note, our considered class of MAFs are far to be confused with the ones  defined in \cite{Elfardi2017a,Elfardi2018prepreint} or in \cite{elgourari2011spectral,ghanmi2011characterization}.
Extension to the orbifold $\Gamma\backslash \C^n $ for given  discrete subgroup $ \Gamma$ of $\U(n)\ltimes \C^n $, acting on $\C^n$ by the mappings $ g.z = Az + b $ for $g=[A,b]\in \U(n)\ltimes \C^n$, i.e., translations and rotations,
is also discussed. We show that an analog of one dimensional lifting theorem fails in higher dimension without additional assumptions on the considered equivariant pair. 
In fact, we provide sufficient and necessary conditions on the equivariant mapping to the corresponding magnetic field be constant.
 
The organization of the paper is as follow. In Section 2, we discuss some basic properties of equivariant pairs $ (\rho,\tau) $, including their characterizations like the one for which  $ \rho $ separates variables ($ \rho=\tilde{\rho}\ltimes\tau $). Section 3 is devoted to give a brief review of the space of mixed automorphic functions, and to provide a necessary and sufficient condition ensuring its non-triviality.
Section~4 deals with the geometrical realization of the magnetic Schr\"odinger operator, and we investigate their variance property with respect to the automorphic action, as well as the connection to the Landau Hamiltonian. 
The lifting theorem is discussed and proved in Section 5, and next used to describe and derive  the concrete spectral analysis of the constructed Laplacian in Section 6. 
  Last but not least, our attempt to generalize the main theorem to higher dimensions is presented in Section~7.

\section{On $G$-equivariant pairs}

 We denote by $G$ the group $ \U(1)\ltimes \C $, that we realize in matrix representation as 
$$ G:= \U(1)\ltimes \C =\set{g:=[a,b] \equiv \left(\begin {array} {c c} a & b \\ 0 & 1\end {array} \right ); \quad a\in \U(1) , ~ b\in \C}, $$
and let $ \Gamma $ be a discrete subgroup of it. The action of $G$ on the complex plane $\C$ is transitive and given by the holomorphic mappings $g\cdot z = a z +b$, where $g=[a,b]\in G$. 

\begin{definition} \label{def:equivariant} By $G$-equivariant pair $(\rho,\tau)$, 	we mean a $G$-endomorphism $\rho$ and a compatible mapping $\tau : \C \to \C$ in the sense that 
	\begin{equation}\label{EquivCond}
	\tau(g\cdot z)= \rho(g)\cdot\tau(z)
	\end{equation}
	holds true for every $ g\in G$ and $z\in\C$.
\end{definition}

The first main objective is an explicit characterization of such equivariant pairs. Thus, it is natural to look for how to exploit \eqref{EquivCond} in order to exhibit suitable maps $ \tau $ for fixed $ \rho $, or vice versa to determinate possible $ \rho $ when $ \tau $ is fixed. Such characterization seems to be far of being realizable without further assumptions. 
Thus, we digress into a partial characterization and provide such bridge between the choice of $ \rho $ and $ \tau $. Namely, we prove the following result.

\begin{theorem}\label{thm:equivariantEquivalence}
	Let $ \rho $ be a G-endomorphism and define $ \Xi_{\rho} := \set{\beta \in \C \ | \ G_{\beta} \supset \rho(G_0)} $, where $ G_x := \set{g\in G \ | \ g\cdot x=x} $. Then, a given mapping $ \tau : \C \longrightarrow \C $ satisfies the equivariant condition \eqref{EquivCond} if and only if
	\begin{equation}\label{eq:taubetaconstruction}
	\tau(z)= \tau_{\beta}(z) := \rho([1,z]) \cdot \beta
	\end{equation}
	for certain fixed $ \beta \in \Xi_{\rho} $ ($ \tau_{\beta}(0) = \beta $).
\end{theorem}

\begin{proof}
	For the proof of ``only if'', notice that Eq. \eqref{EquivCond} implies that $ \rho(G_x) \subset G_{\tau(x)} $ for every fixed $ x\in \C $. As a special case, we have
	\begin{equation}
		\rho(G_0) \subset G_{\tau(0)}.
	\end{equation}
	Next, using the fact that $ G $ acts transitively on $ \C $, one can exhibit for every $ z \in \C $ an element $ g_z \in G $ such that $ z= g_z.0 $. Hence,	one can check that $ \tau(z)= \rho(g_z). \tau(0) $, and therefore we have
	\begin{equation}\label{eq:tauConstruction}
		\tau(z) = \rho([1,z]).\tau(0).
	\end{equation}
	Conversely, it can be checked that \eqref{eq:tauConstruction} defines all mappings from $ \C $ onto $ \C $ that satisfy the equivariant condition \eqref{EquivCond} for a given $ G $-endomorphism $ \rho : \C \longrightarrow \C $.  
	To this end, fix $ \beta \in \Xi_{\rho} $ and note that $ \tau_{\beta} $, as defined by \eqref{eq:tauConstruction}, satisfies
	\begin{equation}\label{eq:taubeta}
	\tau_{\beta}(z) = \rho(g_z) \cdot \beta
	\end{equation}
	for every $ g_z\in G $, such that $ g_z\cdot0=z $. Indeed, since $ g_z=[a,z]=[1,z][a,0] $ for some arbitrary $ a\in \U(1) $, it follows
	$$  \rho(g_z)\cdot\beta=\rho([1,z][a,0])\cdot\beta=\rho([1,z])\cdot\left( \rho([a,0])\cdot\beta \right). $$ Thus, using the fact that $ \beta \in \Xi_{\rho} $, which reads here as $ \rho([a,0])\cdot \beta = \beta $ for every $ a \in \U(1) $, one then gets
	\[ \rho(g_z)\cdot\beta= \rho([1,z])\cdot\beta := \tau_{\beta}(z).  \]
	Now, since for every $ g\in G $ we have $ (gg_z)\cdot 0 = g.z $, we obtain
	\[ \tau(g\cdot z) \stackrel{\eqref{eq:taubeta}}{=} \rho(gg_z)\cdot \beta=\rho(g)\cdot \left(\rho(g_z)\cdot \beta \right) = \rho(g) \cdot \tau_{\beta}(z),  \]
	which means that the equivariant condition \eqref{EquivCond} holds. This completes the proof.
\end{proof}

\begin{remark}
	Note that when $ \rho(G_0) $ is trivial; i.e. $ \rho(G_0)=\set{1} $, we have $ \Xi_{\rho} = \C $, and therefore, for the particular case $ \rho = \operatorname{Id}_G $ ($ \operatorname{Id}_G $ being the identity map of $ G $), the mapping $ \tau $ such that the pair $ (\operatorname{Id}_G,\tau) $ satisfies the equivariant condition \eqref{EquivCond} are the translations $ \tau_{\beta}(z)=z + \beta $ with $  \beta \in \C $. Note also that such translations satisfy \eqref{EquivCond} in the case when the G-endomorphism $ \rho $ verifies $ \rho([1,z])=(a_z,z) $.
\end{remark}

\begin{example}\label{exmp:hammond}
	Let $ \mathfrak{h}=[\alpha_{\mathfrak{h}},\beta_{\mathfrak{h}}] $ be a fixed element in $ G $.
	Consider the $ G $-endomorphism $ \rho_{\mathfrak{h}} $ given by
	\begin{equation}\label{eq:hammondrho}
	\rho_{\mathfrak{h}}(g):= \frakh g\frakh^{-1}=[a,(1-a)\beta_{\frakh}+b\alpha_{\frakh}]; \qquad g=[a,b]\in G,
	\end{equation}
	and the map $  \tau_{\frakh} : \C \longrightarrow \C $ defined by
	\begin{equation}\label{eq:hammondTau}
	\tau_{\frakh}(z):= \frakh\cdot z = \alpha_{\mathfrak{h}} z + \beta_{\mathfrak{h}}.
	\end{equation}
	The pair $ (\rho_{\frakh}, \tau_{\frakh}) $ defines an ``alteration of $ (\operatorname{Id}_G, \operatorname{Id}_{\C}) $'' in the sense of \cite{hammond},
and satisfies the equivariant condition \eqref{EquivCond}. This is in correlation with what precede, since for such $ \rho_{\frakh} $, we have $ \rho_{\frakh}([a,b])\cdot\beta_{\frakh} =\beta_{\frakh} $, i.e. $ \beta_{\frakh}\in \Xi_{\rho} $. Therefore, one can use \eqref{eq:taubetaconstruction} to reproduce again $ \tau_{\frakh} $  given through \eqref{eq:hammondTau}.
\end{example}

\begin{example} \label{exampleConjugate}
	Consider the conjugate endomorphism under $ G $,
	\[ \rho([a,b]) = \overline{[a,b]} := [\bar{a},\bar{b}] .\]
	In this case, one can see easily that  $ \Xi_{\rho} = \set{0} $. According to Theorem \ref{thm:equivariantEquivalence}, $ ([a,b]\!\mapsto\!\overline{[a,b]},\tau) $ is an equivariant pair if and only if :
	\begin{equation}
	\tau(z):=[1,\overline{z}] \cdot 0 = \overline{z}.
	\end{equation} 
\end{example}

We conclude this section by considering the case  of the mapping $ \rho : G \longrightarrow G $  separating variables, in the sense that it takes the form $ \rho = \tilde{\rho} \ltimes \tilde{\tau} $, meaning $ \rho([a,b]) :=[\tilde{\rho}(a),\tilde{\tau}(b)] $, with $ \tilde{\rho} : \U(1)\longrightarrow \U(1) $ and $ \tilde{\tau}: \C \longrightarrow \C $.


\begin{theorem} \label{thm::rauTildTauTild}
	Keep notations as above and assume that the pair $ (\tilde{\rho} \ltimes \tilde{\tau},\tau) $ satisfies the equivariant condition \eqref{EquivCond}. Then, we have two possible cases :
	\begin{itemize}
		\item $ \tilde{\rho} $ is a $ \U(1) $-endomorphism, $\tilde{\tau} $ is additive such that $ \tilde{\tau}(ab+c) = \tilde{\rho}(a)\tilde{\tau}(b)+\tilde{\tau}(c) $, $ a \in \U(1) $, $ b,c \in \C $ and $ \tau(z)= \tilde{\tau}(z)$.
		\item $ \tilde{\rho} \equiv 1, \tilde{\tau} $ is additive such that $ \tilde{\tau}(ab+c) = \tilde{\tau}(b)+\tilde{\tau}(c) $, $ a \in \U(1) $, $ b,c \in \C $ and $ \tau(z)= \tilde{\tau}(z)+\beta $ for certain $ \beta \in \C $.
	\end{itemize}
\end{theorem}

\begin{proof}
Notice first that it is clear that $ \rho $ is a $ G $-endomorphism if and only if $ \tilde{\rho} $ is a $ \U(1) $-endomorphism and $ \tilde{\tau} $ satisfies
\begin{equation}\label{eq:tauTildEquivariance}
\tilde{\tau}(ab+c)= \tilde{\rho}(a)\tilde{\tau}(b)+\tilde{\tau}(c)
\end{equation}
for every $ a \in \U(1) $ and $ b,c \in \C $. 
Therefore, $ \tilde{\tau} $ becomes additive and satisfies
\begin{equation}
\tilde{\tau}(ab)=\tilde{\rho}(a)\tilde{\tau}(b)
\end{equation}
since $  \tilde{\tau}(0)=0 $.
According to Theorem \ref{thm:equivariantEquivalence}, the mappings $ \tau:\C \longrightarrow \C $, satisfying \eqref{EquivCond} relatively to $ \rho = \tilde{\rho} \ltimes \tilde{\tau} $, are those given by
\begin{equation} \label{eq:tauIn2}
\tau_{\beta}(z) = \tilde{\tau}(z)+\beta
\end{equation}
for varying $ \beta \in \Xi_{\rho} $. Moreover, we can show that for all $ \rho = \tilde{\rho} \ltimes \tilde{\tau} $ we have $ 0 \in \Xi_{\rho} $. Then, from \eqref{eq:tauIn2} with $ \beta = 0 $,  we get that the compatible map is $ \tau_{0}(z) = \tilde{\tau}(z) $.
On the other side, it is easy to see that $ \beta \neq 0 $ belongs to $ \Xi_{\rho} $ if and only if $ \tilde{\rho} \equiv 1 $. Indeed,
\[ \rho([a,0]) \cdot \beta := [\tilde{\rho}(a),\tilde{\tau}(0)]\cdot \beta = [\tilde{\rho}(a),0]\cdot \beta = \tilde{\rho}(a)\beta. \]
In this case, $ \rho(G_0) $ is trivial and $ \Xi_{\rho}= \C $. Furthermore, \eqref{eq:tauTildEquivariance} reads simply as
\begin{equation}
\tilde{\tau}(ab+c) = \tilde{\tau}(b)+\tilde{\tau}(c),
\end{equation}
and clearly yields
\begin{equation}\label{eq:tauInRhoTauTild}
\tau_{\beta}(az+b)= \tau_{\beta}(z)+\tau_{\beta}(b)+\beta
\end{equation}
for all $ a\in \U(1) $ and $ z,b \in \C $. 
This ends the proof of Theorem \ref{thm::rauTildTauTild}.
\end{proof}

\begin{remark}
	In the context of Example \ref{exmp:hammond}, we have
	\begin{equation}
	\tau_{\frakh}(az+b) = \tilde{\rho}(a)(\tau_{\frakh}(z)-\tau_{\frakh}(0)) + \tau_{\frakh}(b),
	\end{equation}
	which is clearly similar to \eqref{eq:tauInRhoTauTild}.
\end{remark}

\begin{remark}
	The equivariant pair given in Example \ref{exampleConjugate} by $ ([a,b]\!\mapsto\!\overline{[a,b]},z\!\mapsto\!\overline{z}) $ fits the first point of  Theorem \ref{thm::rauTildTauTild}. In fact we have, for all  $a \in \U(1) $ and $ z\in \C $
	$$\begin{array}{lcr}
	\tilde{\rho}(a)=\bar{a}, \mbox{ and } \tau(z)= \tilde{\tau}(z)=\bar{z}.
	\end{array} $$
\end{remark}

\section{Mixed automorphic functions}

Let $\Gamma$ be a discrete subgroup of $G$, $ (\rho,\tau) $ be an equivariant pair such that $\rho(\Gamma)\subset \Gamma$, and $ \chi $ a map such that $ \chi : \Gamma \rightarrow \U(1) $. Let $ \nu $ and $  \mu $ be two reals such that $ \nu,\mu>0$. 
Associated with data $ (\Gamma;\chi, \nu, \mu, \rho,\tau)$, we perform $ \mixedSAutomorphic(\C) $, the vector space of smooth complex-valued functions $ F $ on $ \C $ satisfying the functional equation:
\begin{equation}\label{eq:AutoFunctional}
F(\gamma\cdot z)= J^{\nu,\mu}_{\rho,\tau}(\gamma,z) F(z) 
\end{equation}
for every $ \gamma \in \Gamma $ and $ z \in \C $, where
\begin{equation}
\label{eq:Mixedfact} J^{\nu,\mu}_{\rho,\tau}(\gamma,z) : = \chi(\gamma) j^\nu(\gamma, z) j^\mu(\rho(\gamma), \tau(z)) 
\end{equation}
and
\begin{equation*} 
j^\alpha(g, z) : = \e^{-\alpha i \Imaginaary\scal{z,\gamma^{-1}\cdot0}}.
\end{equation*}
	The equivariant condition \eqref{EquivCond} ensures that $j^\alpha(\rho(\cdot), \tau(\cdot))$  is also an automorphic factor	when $j^\alpha(\cdot, \cdot)$ is one too. Indeed, for any $  \gamma, \gamma' \in \Gamma $, we have
\begin{align*}
	j^\alpha (\rho(\gamma\gamma'), \tau(z)) &= j^\alpha (\rho(\gamma)\rho(\gamma'), \tau(z)) \\
	& =  j^\alpha (\rho(\gamma), \rho(\gamma')\cdot\tau(z)) j^\alpha (\rho(\gamma'),\tau(z)) \\
	& =  j^\alpha (\rho(\gamma), \tau(\gamma'\cdot z)) j^\alpha (\rho(\gamma'),\tau(z)),
\end{align*}
where in the second line we have  used  the co-cycle equation verified by $j^\alpha$  being an automorphic factor, and the third line follows by means of \eqref{EquivCond}.

\begin{definition}
	The space of $ \mixedSAutomorphic(\C) $ is called the space of planar mixed automorphic functions, or also  of bi weight $ (\nu,\mu) $, with respect to the equivalent pair $ (\rho,\tau) $ and the discrete subgroup $ \Gamma $.
\end{definition}

The space $ \mixedSAutomorphic(\C) $ can be realized as holomorphic sections of the line bundle $\Gamma\backslash (\C\times \C)$  over $\Gamma\backslash \C$ with fiber $\C$ by considering the action
$$\gamma \cdot (z,w) = \left( \gamma \cdot z, J^{\nu,\mu}_{\rho,\tau}(\gamma,z)  w\right)  .$$  
As in the classical setting,  the data $ (\Gamma;\chi, \nu, \mu, \rho,\tau)$ cannot be chosen freely or else the space will be trivial.
Namely, using  similar arguments as in \cite{ghanmi2011characterization,elgourari2011spectral,Elfardi2017a}, we can show the following.

\begin{proposition} \label{prop:nontrivialSecondKind}
	The functional space $ \mathcal{M}^{\nu,\mu}_{\rho,\tau}(\C) $ is nontrivial if and only if the condition
	\begin{equation}\label{eq:RDQtypeCondition}
	\chi(\gamma\gamma')=\chi(\gamma)\chi(\gamma) e^{-2i\phi^{\nu,\mu}_\rho(\gamma,\gamma')} 
	\end{equation}
	holds for every $\gamma,\gamma'\in \Gamma$, where the phase function $\phi^{\nu,\mu}_\rho(g,g')$ is the real-valued function defined on $G\times G$ by
	\begin{equation}\label{PhaseFactor}
	\phi^{\nu,\mu}_\rho(g,g'):= \Imaginaary  \Big(\nu \scal{g^{-1}\cdot
		0,g'\cdot 0} + \mu \scal{\rho({g^{-1}})\cdot
		0,\rho(g')\cdot 0} \Big).
	\end{equation} 
\end{proposition}

\section{An invariant Schr\"odinger operator }

Starting from the next section, we will be interested in the spectral analysis on the space of mixed automorphic functions. To this end, we deal with a specific invariant magnetic Schr\"odinger operator. Recall, for instance, that a magnetic Schr\"odinger operator on a complete oriented Riemannian manifold $(M,g)$ 
can defined by 
\begin{eqnarray}\label{eq:shrodinger1:Chap5}
H_{\theta} =(d+ \operatorname{ext} \ \theta)^{*}(d+\operatorname{ext} \ \theta),
\end{eqnarray}
where $ \theta $ is a given $ \mathcal{C}^{1} $ real differential $ 1 $-form on $ M $ (potential vector),
$d$ stands for the usual exterior derivative acting on the space of differential $p$-forms $ \Omega^p (M)$, $\operatorname{ext}\theta$ is the operator of exterior left multiplication by $\theta$, i.e., $(\operatorname{ext}\theta) \omega = \theta \wedge \omega,$ 
and $(d+ \mbox{ext}\theta)^{*}$ is the formal adjoint of  $d+ \mbox{ext}\theta$ with respect to the Hermitian product
\[ \scalforms{\alpha , \beta} = \int_{M} \alpha \wedge \star \beta \]
induced by the metric  $ g $ on $ \Omega^p (M)$, where $ \star $ denotes the Hodge star operator associated with the volume form. 
In our case, $ M$ is the complex plane $ \C $ equipped with its K\"ahler metric
$
g= ds^2=- ({i}/{2}) dz\otimes d \overbar z= dx\otimes dy
$
and the corresponding volume form is $\Vol=dx dy$.
Thus, associated with the parameters $\mu, \nu, \alpha>0$ and given equivariant pair $(\rho,\tau)$, we consider the  potential vector in \eqref{eq:theta2:Chap5} 
which explicitly reads
\begin{equation}\label{eq:theta2:Chap5}
\theta^{\nu,\mu}_\tau(z) = -\left( \frac{\overline{S^{\nu,\mu}_\tau(z)}}{2}dz - \frac{S^{\nu,\mu}_\tau(z)}{2} d\bar z\right) ,\end{equation}
where 
$$  S^{\nu,\mu}_\tau (z) := \nu z +\mu \Big(\tau \frac{\partial\bar
	\tau}{\partial  \bar z}- \bar \tau \frac{\partial \tau}{\partial
	\bar z} \Big).$$
Straightforward computation shows that the corresponding explicit expression of 
the Schr\"odinger operator $\mixedSDelta  := - H_{\theta^{\nu,\mu}_\tau(z)}$ in \eqref{eq:shrodinger1:Chap5}, for given 
$\theta^{\nu,\mu}_\tau$ in \ref{eq:theta2:Chap5},
is given by 
\begin{equation}
	\mixedSDelta  = 4 \frac{\partial^2 }{\partial z\partial \bar z}    + 2\Big(S^{\nu,\mu}_\tau
	\frac{\partial }{\partial z} -
	\overline{S^{\nu,\mu}_\tau} \frac{\partial }{\partial \bar z} \Big)
- | S^{\nu,\mu}_\tau |^2  + \mu (\tau\Delta \bar \tau -\bar \tau \Delta \tau)
\label{ExplicitAutoLap}
\end{equation}
and generalizes, somehow,  the one given by
\begin{equation}\label{LandauHamiltonian}
\Delta_{\nu}= -H_{\theta_\alpha(z)} =  4
\frac{\partial^2}{\partial z \partial \bar z} + 2\nu
\left( z\frac{\partial}{\partial z}-\bar z \frac{\partial}{\partial \bar
	z}\right)  - {\nu^2} |z|^2 = \Delta^{\nu,0}_\tau,
\end{equation}
 the potential vector $\theta_\nu (z):= {-}\frac{\nu}{2}(\bar{z}dz-zd\bar{z}),$
representing the standard Landau gauge associated with a constant magnetic field of intensity $ \nu$.
It should be mentioned here that, from physics standpoint,  the considered Schr\"odinger operators are the Hamiltonians governing the behavior of a charge in the complex plane under the influence of a  magnetic fields \cite{asch-magnetic}. 
When they are acting on classical automorphic functions, they can be viewed geometrically as the Bochner Laplacian defined on the smooth cross sections of a line bundle over $ \Gamma\backslash \C $.

 The choice of $\theta^{\nu,\mu}_\tau$ is made to ensure that the underling $\mixedSDelta$ be invariant with respect to the projective representation 
 \begin{equation}\label{eq:repreMixed}
 [\mathcal{T}^{\nu,\mu}_gf](z):= \overbar{J^{\nu,\mu}_{\rho,\tau}
 	(g,z)} f(g\cdot z) ,
 \end{equation}
  associated with the mixed automorphic factor given by
\eqref{eq:Mixedfact}. 
This manifests in the following result.

\begin{proposition}\label{prop:invariance}
	The operator $\mixedSDelta$  commutes with $\mathcal{T}^{\nu,\mu}_g$, that is we have 
	\begin{align*}
	\mathcal{T}^{\nu,\mu}_g\mixedSDelta  = \mixedSDelta \mathcal{T}^{\nu,\mu}_g; \, g\in G.
	\end{align*} 

\end{proposition}

\begin{proof}
The proof is similar to the one given in \cite[Theorem~3.4]{elgourari2011spectral}. Though, it involves some more computational difficulties. For instance, we find ourselves in need for the following handy formulas driven from the equivariant condition 
$$\tau(g\cdot z)=\rho(g)\cdot\tau(z)= \phi(g) \tau(z) +\psi(g)$$
for given $G$-endomorphism $\rho: G\to G=\U(1)\ltimes \C$; $\rho(g)=[\phi(g),\psi(g)]$.
Indeed, by differentiating the equivariant condition, we get the identities
	\begin{align}
		& \frac{\partial \tau}{\partial z}(g\cdot z) = \overbar{a}\phi{(g)} \frac{\partial \tau}{z}(\partial z) \label{eq:handyeq1}\\
		&\frac{\partial \overbar{\tau}}{\partial z}(\gamma\cdot z) = \overbar{a}\overbar{\phi{(g)}} \frac{\partial \overbar{\tau}}{\partial z}(z)
		\\
		&\frac{\partial \tau}{\partial \overbar{z}}(g\cdot z) = {a}\phi(g )\frac{\partial \tau}{\partial \overbar{z}}(z) \label{eq:handyeq2}\\ &\frac{\partial \overbar{\tau}}{\partial \overbar{z}}(g\cdot z) = {a}\overbar{\phi(g)} \frac{\partial \overbar{\tau}}{\partial z}(\partial \overbar{z}).\label{eq:handyeq4}
	\end{align} 
Therefore, the transform $ \mathcal{T}^{\nu,\mu}_g $ can be naturally extended to the space of differential forms by considering
\begin{equation}
\mathcal{T}^{\nu,\mu}_g \omega = \overbar{\JiMuNu(g,z)}g^*\omega,
\end{equation}
where $ g^* $; $= \U(1)\ltimes \C$, denotes the pull-back mapping on the space of differential form of $ \C $.
	
	Now, by component-wise straightforward computations, making use of the equations \eqref{eq:handyeq1}--\eqref{eq:handyeq4}, we show that the potential vector $ \theta^{\nu,\mu}_\tau $ in \eqref{eq:theta2:Chap5} satisfies 
	\begin{equation}\label{eq:invariantp}
	g^* \theta_{\nu,\mu}  =\theta_{\nu,\mu} +  \overbar{d  \log (J^{\nu,\mu}(g,z))} .
	\end{equation}
Subsequently, we have
	\begin{align*}
	\mathcal{T}^{\nu,\mu}_g\big((d+ \operatorname{ext} \theta_{\nu,\mu}) f\big)
	&= \overbar{J^{\nu,\mu}(g,z)}  \Big[g^{*}\Big((d+ \operatorname{ext}\theta_{\nu,\mu})f\Big)\Big]\\
	&= \overbar{J^{\nu,\mu}(g,z)}  \Big(d[g^{*}f]+
	[g^{*}\theta_{\nu,\mu} ]\wedge [g^{*}f]\Big)\\
	&\stackrel{\eqref{eq:invariantp}}{=} \overbar{J^{\nu,\mu}(g,z)} d [g^{*}f]+
	\overbar{J^{\nu,\mu}(g,z)}\theta_{\nu,\mu} [g^{*}f] +  d(\overbar{J^{\nu,\mu}(g,z)}) [g^{*}f] \\
	&= d(\overbar{J^{\nu,\mu}(g,z)} [g^{*}f])
	+ \theta_{\nu,\mu} \overbar{J^{\nu,\mu}(g,z)}[g^{*}f]\\
	&= (d+ \operatorname{ext} \theta_{\nu,\mu}) \big(\mathcal{T}^{\nu,\mu}_g f\big).
	\end{align*}
	For these algebraic computations we have made use of the facts $g^{*}d=dg^{*}$ as well as $g^{*}(\alpha\wedge\beta)= g^{*}\alpha\wedge g^{*}\beta,$ together with the identity \eqref{eq:invariantp}. 
	This completes the proof, since $\mathcal{T}^{\nu,\mu}_g$ commutes also with
	$(d+ \operatorname{ext} \theta_{\nu,\mu})^{*}$ for $\mathcal{T}^{\nu,\mu}_g$ being a unitary transformation. 
\end{proof}

The invariance property shows in particular that the eigenvalue problem for  $ \mixedSDelta $ on the space of MAFs $ \mixedSAutomorphic(\C) $ is well defined. Moreover, we may prove that the eigenvalue problem of $ \mixedSDelta $ on $ \mixedSAutomorphic(\C) $ is equivalent to the eigenvalue problem of $ \Delta_{\alpha} $ on the space of  classical automorphic functions $\mathcal{F}^{\alpha}_{\Gamma,\chi}(\C)$ for for some special $ \alpha $ that depends on the parameters $ \nu,\mu $ and the equivariant pairs $ (\rho,\tau) $.
The proof of this assertion is contained Key Lemma \ref{keylem1:eigenproblem} and Theorem \ref{Thm2} below. 

\begin{keylemma}\label{keylem1:eigenproblem}
	Keep notations as above. Then, the following assertions hold trues
	\begin{enumerate} 
\item[i)]	The quantity
	\begin{equation}\label{NewMagnfield}
	\B (z):= \nu +\mu\left(\left|\frac{\partial \tau}{\partial z}({z})\right|^2- \left|\frac{\partial \tau}{\partial\bar z}({z})\right|^2\right)
	\end{equation}
	 is a real constant on $\C$.

\item[ii)]	Set $ \alpha =\alpha_\tau^{\nu,\mu}= \B $. Then, 
	the operators $ \Delta_{\alpha} $ and $ \mixedSDelta $ are intertwining, i.e.,
	\begin{equation} \label{interwin}
		\mixedSDelta=e^{-i\varphi_{\tau}^{\nu,\mu}(z)} \Delta_{\B\!\!\!} \ e^{i\varphi_{\tau}^{\nu,\mu}(z)}
	\end{equation}
	for some smooth mapping $ \varphim $.

\item[iii)]	There exists a smooth real mapping $ \varphim $ such that \eqref{interwin} is fulfilled.  
\end{enumerate}
\end{keylemma}

\begin{proof}	
Using \eqref{eq:handyeq1} and \eqref{eq:handyeq2}, 
as well as the fact that $$ \phi(g) \frac{\partial (g\cdot z)}{\partial z}  \in \U(1) ,$$
 it follows that $ \B $ is invariant under the action of $ G $, and therefore is constant on $ \C $.

	To prove \eqref{interwin}, it is enough to show that the exterior derivatives of their potential vectors are equal. Indeed, it can be checked that $ d\theta^{\nu,\mu}_\tau(z) = d\theta_{\B} (z) $. Therefore, there exists a smooth mapping $ \varphim $ such that
	\begin{equation}
	\theta^{\nu,\mu}_\tau(z) = \theta_{\B} + id\varphim.
	\end{equation}
The map	$ \varphim $ can be freely chosen up to a translation by constants. By means of the last equation we can write
	\begin{equation}
	(d+\operatorname{ext} \ \theta^{\nu,\mu}_\tau) = e^{-i\varphim} (d+ \operatorname{ext} \ \theta_{\B}) e^{i\varphim}.
	\end{equation}
	Accordingly, its adjoint $ (d+ \operatorname{ext}\theta^{\nu,\mu}_\tau)^* $ holds similar equality. Subsequently,
	\begin{equation}
	\mixedSDelta=e^{-i\varphi_{\tau}^{\nu,\mu}(z)} \Delta_{\B\!\!\!} \ e^{i\varphi_{\tau}^{\nu,\mu}(z)}.
	\end{equation}

The last assertion follows easily by making the observation that 
$$2i d \Im(\varphim) = d(\varphim + \overline{\varphim})= 0,$$ since 
$\overline{\theta^{\nu,\mu}_\tau} = - \theta^{\nu,\mu}_\tau$ and $ \overline{\theta_{\B} } =   - \theta_{\B} $. 
Therefore, one can check that $ \varphim $ has a constant purely imaginary part, and then can be translated such that it is real.
\end{proof}

\begin{remark}
	The intertwining property for the Schr\"odinger operators $ \Delta_{\B}  $ and $ \mixedSDelta $ can be argued using physical interpretation (see for example \cite{asch-magnetic}). 
	Indeed, by \eqref{NewMagnfield}, such operators are observables for the quantum behavior of a charged particle under the influence of such constant magnetic field,
	which means that they are related by a gouge transformation. 
	Furthermore, the two operators would rather be unitary equivalent if we had considered a different convention for the Schr\"odinger operator. Namely, using the connection $ (d\pm i\theta) $ instead of $ (d\pm \theta) $, which is common in physic literature.
\end{remark}

\section{Lifting theorem for $ \mixedSAutomorphic(\C) $}

The function $ \varphim $ in $iii)$ of Key Lemma \ref{keylem1:eigenproblem} turns out to be the key to establish, as well, a bijection between  classical automorphic functions and mixed automorphic functions. Indeed, let us define $ \mathcal{W}_{\tau}^{\mu,\nu} $ to be the transformation given by
\begin{equation}\label{TransW}
\left[ \mathcal{W}_{\tau}^{\mu,\nu}f\right] (z) :=  e^{i\varphim(z)} f(z).
\end{equation}

\begin{theorem}\label{Thm2}
The range of the space of mixed automorphic functions 
	$ \mathcal{M}^{\nu,\mu}_{\rho,\tau}(\C) $ by  the transform	$ \mathcal{W}_{\tau}^{\mu,\nu} $ is exactly  the space of classical $(\Gamma,\chi_\tau)$-automorphic functions 
	\begin{equation*}
	\mathcal{F}^{\B}_{\Gamma,\chi_{\tau}}(\C):= \set{F\in\mathcal{C}^\infty(\C), \quad F(\gamma\cdot z)= \chi_\tau(\gamma)j^{\B}(\gamma, z)  F(z)}
	\end{equation*}
	associated to the specific pseudo-character defined on $\Gamma$ by
	\begin{equation}\label{eq:chiTau0}
\chi_\tau(\gamma)=\chi(\gamma)\exp\left( {\varphim(\gamma\cdot0)-2i\mu  \Imaginaary \scal{\tau(0),\rho(\gamma)^{-1}\cdot 0}}\right) .
\end{equation}
\end{theorem}

For the proof,  we begin with the following assertion concerning the function $\widehat{\chi_\tau}$ defined on $\C \times \Gamma$ by
\begin{equation}\label{eq:chiTau}
\widehat{\chi_\tau}(z;\gamma):=  e^{(\varphim(\gamma\cdot z)- \varphim(z))} \chi(\gamma) j^{\nu-\B}(\gamma,z) j^{\mu}(\rho(\gamma),\tau(z)).
\end{equation} 
 Here
$ j^{\alpha}= e^{-i\alpha  \Imaginaary \scal{z,g^{-1}\cdot 0}} $. 

\begin{keylemma}\label{LemTransMixed}
	The function $\widehat{\chi_\tau}$ 	is independent of the variable $z$ and we have 
		\begin{equation}\label{eq:chiTauchi}
	 \widehat{\chi_\tau}(z;\gamma)= \chi_\tau(\gamma). 
	\end{equation}
\end{keylemma}

\begin{proof} Differentiation of $\widehat{\chi_\tau}(z;\gamma)$ with respect to the $z$-variable gives
	\begin{align} \label{Indep1}
	 \frac{\partial \log \widehat{\chi_\tau}}{\partial z}   & =  S_1 + S_2,
	\end{align}
	where $S_1$ and $S_2$ stand for
		\begin{align*} 
		S_1 : =   \frac{\partial \varphim}{\partial z}(\gamma\cdot z)- \frac{\partial\varphim}{\partial z}(z)   
		\end{align*}
and 	
		\begin{align*}   S_2:=  [\nu-\B] \frac{\partial i \Imaginaary\scal{z,\gamma^{-1}\cdot 0}}{\partial z}  + \mu   \frac{\partial i \Imaginaary\scal{\tau(z),\rho(\gamma)^{-1}\cdot 0}}{\partial z}
		.\end{align*}
Thus, it is easy to check that
	\begin{equation}\label{IdI}
	S_2 = \frac{1}{2}\left[ (\B-\nu)a\overbar{b}-\mu \Big( \phi(\gamma)\overbar{\psi(\gamma)}\frac{\partial \tau}{\partial z}{} - \overbar{\phi(\gamma)}{\psi(\gamma)}\frac{\partial \overbar \tau}{\partial z}{} \Big) \right].
	\end{equation}
For the explicit computation, we have used (here and elsewhere) the notation $ \gamma:=[a,b] $ and $\rho(\gamma)=[\phi(\gamma),\psi(\gamma)]$.
For the evaluation of the component $S_1$, we make use of  
	\begin{align}
	\frac{\partial \varphim}{\partial z}(z) & = \frac{1}{2}\left[  \B\overbar{z} -  S^{\nu,\mu}_\tau (z) \right] = \overline{\left( \frac{\partial \varphim}{\partial \overbar{z}}(z)\right) } \label{eq:varphimz} 
\end{align}
which follows by  direct computation, to get
	\begin{align*}
	&\frac{\partial \varphim(\gamma\cdot z)}{\partial z} =\frac{a}{2}  \left[ (\B-\nu) \overbar{(\gamma\cdot z)}-\mu \Big(\overbar{\tau(\gamma\cdot z)} \frac{\partial
		\tau}{\partial z}(\gamma\cdot z)-  \tau(\gamma\cdot z) \frac{\partial \overbar{\tau}}{\partial
		z}(\gamma\cdot z) \Big) \right].
	\end{align*}
Now, direct computations, mainly by making use of the equations \eqref{eq:handyeq1}--\eqref{eq:handyeq4}, infers
\begin{align*}
\frac{\partial \varphim(\gamma\cdot z)}{\partial z}	
	&=\frac{\B-\nu}{2} \left( \overbar{z}+a\overbar{b}\right) \\
	& \quad -  \frac{\mu}{2} \left( \overbar{\tau{}} \frac{\partial \tau}{\partial z} + \phi(\gamma)\overbar{\psi(\gamma)}\frac{\partial \tau}{\partial z} -
	( {\tau( z)} \frac{\partial \overbar \tau}{\partial z}{} + \overbar{\phi(\gamma)}{\psi(\gamma)}\frac{\partial \overbar \tau}{\partial z}{} ) \right),
	\end{align*}
which reduces further to 
\begin{align*}
\frac{\partial \varphim(\gamma\cdot z)}{\partial z}	
	&= \frac{\partial \varphim}{\partial z}(z) +
	\frac{1}{2}\left[ (\B-\nu)a\overbar{b}-\mu \Big( \phi(\gamma)\overbar{\psi(\gamma)}\frac{\partial \tau}{\partial z}{} - \overbar{\phi(\gamma)}{\psi(\gamma)}\frac{\partial \overbar \tau}{\partial z} (z) \Big) \right]\\
	&= \frac{\partial \varphim}{\partial z}(z) +   S_2,
	\end{align*}
	thanks to \eqref{IdI}. This shows that  	$S_1 = S_2 $ and therefore
$	{\partial \widehat{\chi_\tau}}/{\partial z} = 0.$ 
	Moreover, using the facts that 
	$ \overbar{j^\alpha(\gamma,z)} = j^{-\alpha}(\gamma,z) $ and $$ \frac{\partial \varphim(\gamma\cdot z)}{\partial \overbar{z}} = \overbar{\frac{\partial \overbar{\varphim(\gamma\cdot z)}}{\partial z}} = \overbar{\frac{\partial \varphim(\gamma\cdot z)}{\partial {z}}} ,$$ it  follows that
$ {\partial \widehat{\chi_\tau}}/{\partial \overbar{z}} = 0.$ 
	This proves that $\widehat{\chi_\tau}$ is independent of the variable $z$ and therefore we have 
	\begin{equation}\label{eq:chiTauchi}
	\widehat{\chi_\tau}(z;\gamma)= \widehat{\chi_\tau}(z;0) = \chi_\tau(\gamma). 
	\end{equation}
	This ends the proof.
\end{proof}

\begin{proof}[Proof of Theorem \ref{Thm2}]
	We shall prove that $\mathcal{W}^{\nu,\mu}_\tau F$ belongs to $\mathcal{F}^{B^{\nu,\mu}_{\tau}}_{\Gamma,\chi_{\tau}}(\C)$
	whenever  $F \in \mathcal{M}^{\nu,\mu}_{\tau}(\C)$, where $\chi_\tau(\gamma)$ is the one defined through \eqref{eq:chiTau0}. Indeed, for fixed 
	 $ \gamma \in \Gamma $, we have
	\begin{align*}
	[\mathcal{W}^{\nu,\mu}_\tau F](\gamma\cdot z)
	& := e^{i\varphim(\gamma\cdot z)} F(\gamma\cdot z) \\
	&  = e^{i\varphim(\gamma\cdot z)} \chi(\gamma) j^{\nu}(\gamma, z)j^{\mu}(\rho(\gamma), \tau(z)) F(z) \\
	&  = e^{i(\varphim(\gamma\cdot z)-\varphim(z))} \chi(\gamma) j^{\nu}(\gamma, z)j^{\mu}(\rho(\gamma), \tau(z)) [\mathcal{W}^{\nu,\mu}_\tau F](z) \\
	&  = \widehat{\chi_\tau}(z;\gamma)  j^{\B}(\gamma, z) [\mathcal{W}^{\nu,\mu}_\tau F](z),
	\end{align*}
	where $ \widehat{\chi_\tau}(z;\gamma) $ is defined by \eqref{eq:chiTau}.
	This completes the proof thanks to  the Key Lemma  \ref{LemTransMixed}, $\widehat{\chi_\tau}(z;\gamma)= \chi_\tau(\gamma)$. 
\end{proof}

\begin{remark} \label{CorRDQ}
	For $ \mixedSAutomorphic $ being nontrivial, which is equivalent to $\chi$ be a pseudo-character (Proposition~\ref{prop:nontrivialSecondKind}), the data $(\Gamma;\B,\chi_\tau)$ satisfies the  Riemann-Dirac quantization type condition provided in \cite{ghanmiintissar} ensuring that the corresponding space of classical automorphic functions is nontrivial. Here
	$\B$ and $\chi_\tau$ are those defined by \eqref{eq:chiTau0} and \eqref{NewMagnfield}, respectively. 
\end{remark}

\section{Applications: Spectral analysis of $ L^2$-mixed automorphic functions} 

In the sequel, we provide immediate application of the lifting Theorem \ref{Thm2}. Namely, we are concerned with the concrete description of the spectral properties of $\mixedSDelta$ acting on  mixed automorphic functions belonging to the Hilbert space $ L^{2,\nu}_{\mu,\tau}(\C/\Gamma)$ of square integrable functions with respect to  the scalar product 
$$ \scal{f,g}_{\nu,\mu,\tau} := \int_{\C/\Gamma} f(z) \overline{g(z)} e^{-\nu|z|^2 - \mu (|E\tau(z)|^2  - |\overline{E}\tau(z)|^2 )} d\lambda(z)$$
where $E =z \frac{\partial}{\partial z}$ is the complex Euler operator. 
To this end, we begin by determining the range of the eigenspace  of all eigenfunctions of $\mixedSDelta$ in $\mathcal{M}^{\nu,\mu}_{\tau}(\C)$ corresponding to the eigenvalue $\lambda$, 
 \begin{equation}\label{M-Eigenspace}
	\mathcal{E}^{\nu,\mu}_{\tau;\lambda}:= \set{F\in
		\mathcal{M}^{\nu,\mu}_{\tau}(\C); \quad
		\mixedSDelta F = \lambda F}.
\end{equation}
 
\begin{proposition} \label{propMF} 
	We have
	$$ \mathcal{W}^{\nu,\mu}_\tau (\mathcal{E}^{\nu,\mu}_{\tau;\lambda}) = \set{F\in
		\mathcal{F}^{\B}_{\Gamma,\chi_\tau}(\C); \quad
		\Delta_{\B\!\!\!} F = \lambda  F}=: \mathcal{E}^{\B}_{\lambda}. $$
\end{proposition}

\begin{proof}
	The proof 
	follows using $ii)$ and $iii)$ (intertwining) in Key Lemma  \ref{keylem1:eigenproblem} together with the lifting Theorem \ref{Thm2} (invariance).
	\end{proof}


The next result involves the $L^2$ eigenspaces of $\mixedSDelta $
defined by 
$$ A^{2,\nu,\mu,\tau}_{\lambda}(\C) = \mathcal{E}^{\nu,\mu}_{\tau;\lambda} \cap L^{2,\nu}_{\mu,\tau}(\C/\Gamma) 
.$$

\begin{proposition} \label{} 
	The range of the restriction of the transform $\mathcal{W}_{\tau}^{\mu,\nu}$ in   \eqref{TransW} to $A^{2,\nu,\mu,\tau}_{\lambda}(\C)$ is exactly the space of  $L^2$-$\Gamma$-automorphic functions $\mathcal{E}^{\B}_{\lambda} \cap L^{2,\B}_{0,Id_{\C}}(\C/\Gamma) $. In particular, the spectrum of $\mixedSDelta $ is purely discrete and given by the Landau levels $\lambda_k=-2\B(2k+1)$; $k=0,1,2, \cdots$. 
\end{proposition}

\begin{proof}
	The proof follows by means of Proposition \ref{propMF}  together with the fact that the transform $ \mathcal{W}_{\tau}^{\mu,\nu}$  defines an isometric mapping from $L^{2,\nu}_{\mu,\tau}(\C/\Gamma)$ into $L^{2,\nu}(\C/\Gamma)$, up to multiplication constant, and $iii)$ in Key Lemma \ref{keylem1:eigenproblem}.
\end{proof}

Subsequently, for rank one of the discrete subgroup that we identify to $\Gamma =\Z$,
so that $\mathbb{C}/\Z)$ is a strip, the space $L^{2,\nu}_{\mu,\tau}(\C/\Gamma) $  possesses a Hilbertian orthogonal decomposition in terms of $A^{2,\B,\mu,\tau}_{\lambda_k}(\C)$ with $\lambda_k= -2\B(2k+1)$; $k=0,1,2, \cdots$. 
An orthogonal basis of each $A^{2,\nu,\mu,\tau}_{\lambda_k}(\C)$ can be shown to be given in terms of $m^{th}$ Hermite polynomial $H_m$ by the functions 
$$\psi_{m,n}^{\alpha,\B}(z,\overline{z}) = 
e^{i\varphim(z)} \psi_n^{\alpha,\B} (z)
H_m \left(({2\B})^{1/2} \Im(z) +  ({{2}/{\B}})^{1/2} \pi (n+\alpha) \right),
$$
for varying $m\in \Z^+$ and $m\in \Z$, where we have set
$$
\psi_n^{\alpha,\B} (z) :=  
e^{-\frac{\pi^2}{\B}(n+\alpha)^2 }
e^{\frac{\B}2 z^2  + 2i\pi (\alpha + n) z};
\quad n\in\Z.
$$

We conclude this section by providing the concrete description of spectral properties of the Laplacian $\mixedSDelta$ on the free Hilbert space $L^2(\C; d\lambda)$.  

\begin{theorem} \label{SpecProperties} 
	The point spectrum of $\mixedSDelta$ acting on $L^2(\C; d\lambda)$ is discrete and reduces to the Landau levels $\lambda_k=-2\B(2k+1)$; $k=0,1,\cdots$. Moreover, the corresponding 
	$L^2$ eigenspaces 
	$$A^{2}_{k}(\mixedSDelta)=\set{f\in L^2(\C; d\lambda);\quad \mixedSDelta f = \lambda_k  f}$$ leads to the orthogonal decomposition of $ L^2(\C;d\lambda)$ 
	and the explicit closed expression of their eigenprojector kernel $K^{\nu,\mu}_{\tau;k}$ are given by 
	\[K^{\nu,\mu}_{\tau;k}(z,w)= \frac{\B}{\pi} e^{\psi^{\nu,\mu}_\tau(z,w)}
	e^{i \B\Imaginaary \scal{z,w}} e^{-\frac{\B}{2}|z-w|^2} L_k(2\B  |z-w|^2)
	,\] where $L_k(x)=L_k^0(x)$ denotes the usual Laguerre polynomial.
\end{theorem}

\begin{proof}
	This is immediate thanks  to $ii)$ and $iii)$ in the Key Lemma \ref{keylem1:eigenproblem}, combined with  the observation that 
	the transform $ \mathcal{W}_{\tau}^{\mu,\nu}$ in \eqref{TransW} defines an isometric mapping of $L^2(\C; d\lambda)$, up to a multiplicative constant, so that  $\mixedSDelta f = \lambda  f $ for $f\in L^2(\C; d\lambda)$ becomes equivalent to  $ \Delta_{\B\!\!\!} \ g = \lambda g$   for $g=e^{i\varphim(z)} f\in L^2(\C; d\lambda)$. Moreover, if $K_{\B}$ denotes the reproducing kernel of the generalized Bargmann space  (\cite{ghanmiintissar}), we get 
	$$ e^{i\varphim(z)} f(z) =\int_{\C}  K_{\B}(z,w) e^{i\varphim(w)} f(w) d\lambda(w),$$
	which implies that $e^{-i\varphim(z)} K_{\B}(z,w) e^{i\varphim(w)} $ is the reproducing of  $A^{2}_{\lambda_k}(\mixedSDelta)$.
	We conclude for Theorem \ref{SpecProperties}  by making appeal to the results giving the spectrum and closed expression of $K_{\B}(z,w)$ in \cite{ghanmiintissar}. 
\end{proof}

\begin{remark}
	The eigenprojector kernel of $A^{2}_{k}(\mixedSDelta)$ satisfies the invariance property
	\begin{equation*}
		K^{\nu,\mu}_{\tau;k}(z,w)=
		e^{(\psi^{\nu,\mu}_\tau(g\cdot z,g\cdot w) \psi^{\nu,\mu}_\tau(z,w))} e^{i \B\Imaginaary (z-w,g^{-1}.0)}K^{\nu,\mu}_{\tau;k}(g.z,g.w); \, g\in G,
	\end{equation*}
	where $\psi^{\nu,\mu}_\tau(z,w):=\varphim(z)-\varphim(w)$.
\end{remark}

\section{Lifting theorem in high dimensions}

In the previous section, a one-dimensional lifting theorem was proved connecting mixed automorphic functions to classical automorphic functions.
It turns out that a generalization to high dimension fails in general. Below, we provide necessary and sufficient conditions on the equivariant map toward a lifting result in high dimensions. Thus, following the same scheme as in the 1-dimensional, the different constructions for mixed automorphic functions and the corresponding invariant magnetic Laplacians remain valid in high dimension. A discrete subgroup of the semi-direct group $ G = \U(n)\ltimes\C^n $, of the unitary group and $\C^n$, acts  by $ \gamma \cdot z = A z + b$ on $\C^n$,
for $\gamma:=[A,b]\in G$. Such action 
can be expanded component-wise as 
$ (\gamma\cdot z)_i = \sum_{j=1}^{n} a_{i,j} z_j + b_i,$
with $A:=(a_{i,j})_{1\leq i,j\leq n} \in U(n) $, $ b:=(b_i)_{1\leq i\leq n} \in\C^n$, and  $ z:=(z_1,\cdots,z_n)\in\C^n $.
As expected, the mixed automorphic functions 
$$ F(A z + b ) = J^{\nu,\mu}_{\rho,\tau}(\gamma,z) F(z)  = \chi(\gamma) j^\nu(\gamma, z) j^\mu(\rho(\gamma), \tau(z))  F(z)  $$
corresponds to the  equivariant pair $(\rho,\tau)$ such that  $\rho: G \longrightarrow  G$  is a $G$-endomorphism verifying  $\rho(\Gamma) \subset \Gamma$ for given discrete subgroup  $ \Gamma $ of $ G $,  and $\tau : \C^n \longrightarrow  \C^n$  a compatible $ \mathcal{C}^1 $ mapping  such that
\begin{equation}
	\label{eq:chap7:EquivCond}
	 \tau(g\cdot z)= \rho(g)\cdot\tau(z); \qquad g\in G, \quad z\in\C^n,
\end{equation}
with $\tau(z)= (\tau_1(z),\tau_2(z),\cdots,\tau_n(z))$.
In terms of the linear and translational counterparts of $\rho: \gamma \mapsto [\alpha(\gamma),\beta(\gamma)]$ with $ \alpha:=(\alpha_{i,j})_{i,j} $ is a matrix in $U(n)$  and $ \beta:=(\beta_i)_i $ is a column vector in $ \C^n $,
the condition \eqref{eq:chap7:EquivCond} 
reads equivalently as
 \begin{equation}\label{eq:equivarianceComponentsHighDim}
	\tau_\ell = \sum_{j=1}^{n} \alpha_{\ell,j}\tau_j+\beta_\ell.
\end{equation}
The components $\alpha_{i,j}$ and $\beta_{j}$ are functions in $ \gamma $.

The magnetic Schr\"odinger operator $
H_{\theta^{\nu,\mu}_\tau} =(d+ \operatorname{ext} \ \theta^{\nu,\mu}_\tau)^{*}(d+ \operatorname{ext} \ \theta^{\nu,\mu}_\tau) $ corresponding to the potential vector form 
\begin{equation}\label{eq:theta2}
	\theta^{\nu,\mu}_\tau : =-\frac{1}{2} \sum_{\ell=1}^{n}\Big\{\nu(\bar z_\ell dz_\ell -  z_\ell d\bar z_\ell) +\mu(\overline{\tau_\ell}d\tau_\ell - \tau_\ell d\overline{\tau_\ell}) \Big\} 
\end{equation}
is the one leaving invariant the space of mixed automorphic functions.   
The computations are quite straightforward and make use of the facts
\begin{equation}\label{keyx}
	\sum_{j=1}^{n} a_{j,k} \dfrac{\partial \tau_\ell}{\partial z_j}(\gamma\cdot z)=\sum_{j=1}^{n} \alpha_{\ell,j} \frac{\partial \tau_j}{\partial z_k}
	\quad \mbox{and} \quad 
	\sum_{j=1}^{n} \overbar{a_{j,k}} \dfrac{\partial \tau_\ell}{\partial \overbar{z_j}}(\gamma\cdot z)= \sum_{j=1}^{n} \alpha_{\ell,j} \frac{\partial \tau_j}{\partial \overbar{z_k}}
\end{equation}
for all $ k,\ell=1,2,\cdots,n $,
which follow by identification after differentiating the right-hand sides of the identities  obtained by differentiating left-hand side of \eqref{eq:equivarianceComponentsHighDim} and with the usage of chain rule, for functions of several complex variables.
However, the description of spectral analysis of the constructed Laplacian can not be handled by adopting similar approach as for one dimensional case. 
Here the corresponding magnetic field is not necessary constant, which was crucial in establishing the lifting Theorem \ref{Thm2}. 
We claim that the intertwining property between $ H_{\theta^{\nu,\mu}_\tau}  $ and $ \Delta_{\nu} $ does not hold in general in high dimensions. More precisely, we prove

\begin{theorem}
	The intertwining property for  $ H_{\theta^{\nu,\mu}_\tau}  $ and $ \Delta_{\nu} $ holds true if and only if and only if $  \abs{\partial_{\overbar{z_k}}\tau_\ell}^2 - \abs{\partial_{{z_k}}\tau_\ell}^2$ is constant and
	$$\det \left(\begin{array}{cc}
		\partial_{z_i}\overbar{\tau_\ell}  &    \partial_{z_j}\overbar{\tau_\ell} \\ 
		\partial_{z_i}{\tau_\ell}  & \partial_{z_j}{\tau_\ell}
	\end{array}\right)  = 
	\det \left(\begin{array}{cc}
		\partial_{z_i}\overbar{\tau_\ell}  &    \partial_{\overbar{z_j}}\overbar{\tau_\ell} \\ \partial_{z_i}{\tau_\ell}  &  \partial_{\overbar{z_j}}{\tau_\ell}
	\end{array}\right) =0 $$
for every $i,j,k,\ell=1,2, \cdots$. 
\end{theorem}

\begin{proof}
	For the proof, let us focus on the factor $\omega_\tau(z):=d(\overline{\tau_\ell}d\tau_\ell - \tau_\ell d\overline{\tau_\ell})$, where $ \tau $ is present.  
It will be of importance in the squeal to express the information within the equivariant condition~\eqref{eq:chap7:EquivCond} on the partial derivatives of the components of $ \tau $. 
Therefore, we get  
	\begin{align*}
	\omega_\tau(z) 
	&= \sum_{\ell=1}^{n} \sum_{k=1}^{n} F_{\tau_\ell,k} (z) d\overbar{z_k}\wedge d\overbar{z_k}   
	 \\&+ 
	\sum_{\ell=1}^{n} \Big( 	\sum_{j=1}^{n}\sum_{i=1}^{j-1}  A_{\tau_\ell,ij} dz_i\wedge dz_j 
	+  B_{\tau_\ell,ij} dz_i\wedge d\overbar{z_j} + C_{\tau_\ell,ij} d\overbar{z_i}\wedge dz_j 
	+  D_{\tau_\ell,ij}  d\overbar{z_i}\wedge d\overbar{z_j} 
	\Big).
\end{align*}
where the quantities $A_{\tau_\ell,ij}$, $B_{\tau_\ell,ij}$, $C_{\tau_\ell,ij}$, $D_{\tau_\ell,ij}$ and $F_{\tau_\ell,k}$ are given respectively by 
$$ A_{\tau_\ell,ij}:=\partial_{z_i}\overbar{\tau_\ell} \partial_{z_j}{\tau_\ell}
- \partial_{z_j}\overbar{\tau_\ell} \partial_{z_i}{\tau_\ell}
=  \det \left(\begin{array}{cc}
	\partial_{z_i}\overbar{\tau_\ell}  &    \partial_{z_j}\overbar{\tau_\ell} \\ 
	\partial_{z_i}{\tau_\ell}  & \partial_{z_j}{\tau_\ell}
\end{array}\right)  
$$
$$
B_{\tau_\ell,ij}:= \partial_{z_i}\overbar{\tau_\ell} \partial_{\overbar{z_j}}{\tau_\ell}
- \partial_{\overbar{z_j}}\overbar{\tau_\ell} \partial_{z_i}{\tau_\ell}
= \det \left(\begin{array}{cc}
	\partial_{z_i}\overbar{\tau_\ell}  &    \partial_{\overbar{z_j}}\overbar{\tau_\ell} \\ \partial_{z_i}{\tau_\ell}  &  \partial_{\overbar{z_j}}{\tau_\ell}
\end{array}\right)
$$
$$C_{\tau_\ell,ij} :=  \partial_{\overbar{z_i}}\overbar{\tau_\ell} \partial_{z_j}{\tau_\ell}
- \partial_{z_j}\overbar{\tau_\ell} \partial_{\overbar{z_i}}{\tau_\ell}
= \det \left(\begin{array}{cc}	\partial_{\overbar{z_i}}\overbar{\tau_\ell}   &    \partial_{z_j}\overbar{\tau_\ell} \\ \partial_{\overbar{z_i}}{\tau_\ell}  &  \partial_{z_j}{\tau_\ell}
\end{array}\right)
$$
$$ D_{\tau_\ell,ij} := \partial_{\overbar{z_i}}\overbar{\tau_\ell} \partial_{\overbar{z_j}}{\tau_\ell}
- \partial_{\overbar{z_j}}\overbar{\tau_\ell} \partial_{\overbar{z_i}}{\tau_\ell}
=  \det \left(\begin{array}{cc}
\partial_{\overbar{z_i}}\overbar{\tau_\ell}   &    \partial_{\overbar{z_j}}\overbar{\tau_\ell} \\ \partial_{\overbar{z_i}}{\tau_\ell}  &   \partial_{\overbar{z_j}}{\tau_\ell}
\end{array}\right) 
$$
$$ F_{\tau_\ell,k} := \abs{\partial_{\overbar{z_k}}\tau_\ell}^2 - \abs{\partial_{{z_k}}\tau_\ell}^2
=
 \det \left(\begin{array}{cc}
	\partial_{{z_k}}\overbar{\tau_\ell}   &    \partial_{\overbar{z_k}}\overbar{\tau_\ell} \\ \partial_{{z_k}}{\tau_\ell}  &   \partial_{\overbar{z_k}}{\tau_\ell}
\end{array}\right).$$
It is worth to observe that 
$ D_{\tau_\ell,ij} = - \overline{A_{\tau_\ell,ij}}$ and $C_{\tau_\ell,ij} = - \overline{B_{\tau_\ell,ij}}$. 
Therefore,  the operator $ H_{\theta^{\nu,\mu}_\tau}  $ is associated with a constant magnetic field if and only if $ d\theta^{\nu,\mu}_\tau $ is a K\"ahler 2-form. This is equivalent to $A_{\tau_\ell,ij}= B_{\tau_\ell,ij}=0$ and $F_{\tau_\ell,k}$ is constant.
\end{proof}

\begin{corollary}
	The Lifting theorem holds true for holomorphic equivariant mapping $\tau$ such that $\abs{\partial_{{z_j}}\tau_\ell}^2$ is constant and   $ \partial_{\overbar{z_j}}\overbar{\tau_\ell} \partial_{z_i}{\tau_\ell}=0 $ for every $j,k,\ell=1,2, \cdots$. 
\end{corollary}

\bibliographystyle{abbrv}

\begin{thebibliography}{10}
	
	\bibitem{asch-magnetic}
	J.~Asch, H.~Over, and R.~Seiler.
	\newblock Magnetic bloch analysis and bochner laplacians.
	\newblock {\em Journal of Geometry and Physics}, 13(3):275--288, 1994.
	
	\bibitem{Elfardi2017a}
	A.~El~Fardi, A.~Ghanmi, and A.~Intissar.
	\newblock On concrete spectral properties of a twisted laplacian associated
	with a central extension of the real heisenberg group.
	\newblock {\em Advances in Mathematical Physics}, 2017(1), 2017.
	
	\bibitem{Elfardi2018prepreint}
	A.~El~Fardi, A.~Ghanmi, and A.~Intissar.
	\newblock Concrete $l^2$-spectral analysis of a bi-weighted
	$\gamma$-automorphic twisted laplacian associated to berry phase.
	\newblock preprint, 2018.
	
	\bibitem{elgourari2011spectral}
	A.~El~Gourari and A.~Ghanmi.
	\newblock Spectral analysis on planar mixed automorphic forms.
	\newblock {\em Journal of Mathematical Analysis and Applications},
	383(2):474--481, 2011.
	
	\bibitem{ghanmi2011characterization}
	A.~Ghanmi.
	\newblock A characterization of planar mixed automorphic forms.
	\newblock {\em International Journal of Mathematics and Mathematical Sciences},
	2011, 2011.
	
	\bibitem{ghanmiintissar}
	A.~Ghanmi and A.~Intissar.
	\newblock 	Landau automorphic functions on ${\bf C}^n$ of magnitude $\nu$.
	\newblock {\em  J. Math. Phys.}, 49 (2008), no. 8, 083503, 20 pp.
	
	\bibitem{hammond}
	W.~F. Hammond.
	\newblock The modular groups of hilbert and siegel.
	\newblock {\em American Journal of Mathematics}, 88(2):497--516, 1966.
	
	\bibitem{hunt1985mixed}
	B.~Hunt and W.~Meyer.
	\newblock Mixed automorphic forms and invariants of elliptic surfaces.
	\newblock {\em Mathematische Annalen}, 271(1):53--80, 1985.
	
	\bibitem{kaup1983holomorphic}
	L.~Kaup and B.~Kaup.
	\newblock {\em Holomorphic functions of several variables: an introduction to
		the fundamental theory}.
	\newblock Number~3 in De Gruyter Studies in Mathematics. De Gruyter, 1983.
	
	\bibitem{lee-mixed}
	M.~H. Lee.
	\newblock {\em Mixed automorphic forms, torus bundles, and Jacobi forms}.
	\newblock Number 1845 in Lecture Notes in Mathematics. Springer-Verlag Berlin
	Heidelberg, 2004.
	
	\bibitem{stiller-special}
	P.~Stiller.
	\newblock {\em Special values of Dirichlet series, monodromy, and the periods
		of automorphic forms}.
	\newblock Number 299 in Memoirs of the American Mathematical Society. American
	Mathematical Society, 1984.
	
\end{thebibliography}

\end{document}